\documentclass{article}

\usepackage{amsmath}
\usepackage{amsfonts}
\usepackage{amssymb}
\usepackage{latexsym}
\usepackage{amsthm}
\usepackage[margin=1.1in]{geometry}
\usepackage{tikz}
\usepackage{hyperref}

\numberwithin{equation}{subsection}

\newtheorem{theorem}{Theorem}[section]
\newtheorem{corollary}[theorem]{Corollary}
\newtheorem{lemma}[theorem]{Lemma}
\newtheorem{remark}[theorem]{Remark}

\theoremstyle{plain}

\newcommand{\supp}{\operatorname{supp}}
\newcommand{\jdup}{\operatorname{jdup}}
\newcommand{\dup}{\operatorname{dup}}
\newcommand{\for}{\textrm{for}\,\,}

\title{Sparsity of Graphs that Allow Two Distinct Eigenvalues}

\author{Wayne Barrett
\thanks{Department of Mathematics, Brigham Young University, Provo, UT (wb@mathematics.byu.edu) }
\and Shaun Fallat
\thanks{Department of Mathematics and Statistics, University of Regina, Regina, Saskatchewan, CA (shaun.fallat@uregina.ca)}
\and Veronika Furst \thanks{Department of Mathematics, Fort Lewis College, Durango, CO (furst\_v@fortlewis.edu)}
\and Franklin Kenter \thanks{Department of Mathematics, U.S. Naval Academy, Annapolis, MD (kenter@usna.edu)}
\and Shahla Nasserasr \thanks{School of Mathematical Sciences, Rochester Institute of Technology, Rochester, NY (sxnsma@rit.edu)}
\and Brendan Rooney \thanks{School of Mathematical Sciences, Rochester Institute of Technology, Rochester, NY (brsma@rit.edu)} ~\thanks{Corresponding Author}
\and Michael Tait \thanks{Department of Mathematics \& Statistics, Villanova University, Villanova, PA (michael.tait@villanova.edu)}
\and Hein van der Holst \thanks{Department of Mathematics and Statistics, Georgia State University, Atlanta, GA (hvanderholst@gsu.edu)}
}

\date{\today}

\begin{document}

\maketitle

\begin{abstract}
The parameter $q(G)$ of a graph $G$ is the minimum number of distinct eigenvalues over the family of symmetric matrices described by $G$. It is shown that the minimum number of edges necessary for a connected graph $G$ to have $q(G)=2$ is $2n-4$ if $n$ is even, and 
$2n-3$ if $n$ is odd. In addition, a characterization of graphs for which equality is achieved in either case is given. 
\end{abstract}



\section{Introduction} 

For a simple finite undirected graph $G=(V(G), E(G))$ on $n$ vertices we consider the set of $n\times n$ real symmetric matrices corresponding to $G$ as follows: 
\[\mathcal{S}(G)=\{A=[a_{i,j}]\,:\, a_{i,j}=0\,\, \textrm{for}\,\, i\neq j \iff \{i,j\}\notin E(G)\}.\]
The diagonal entries of $A$ have no restrictions, and the off-diagonal entries match the zero-pattern of the adjacency matrix $A(G)$ of $G$. The \emph{Inverse Eigenvalue Problem for Graphs (IEPG)} asks which spectra can occur for matrices $A\in \mathcal{S}(G)$. This is a difficult problem for almost all graphs. A subproblem of the IEPG is to study the multiplicities of the eigenvalues of matrices in $\mathcal{S}(G)$. 
If a symmetric $n\times n$ matrix $A$ has eigenvalues $\lambda_1\leq  \lambda_2\leq  \cdots\leq  \lambda_k$ with multiplicities $m_1, m_2,\ldots, m_k$, respectively, we say $A$ achieves the ordered multiplicity list $(m_1, m_2,\ldots, m_k)$. We write $\sigma(A)=\{\lambda_1^{(m_1)},\lambda_2^{(m_2)},\ldots, \lambda_k^{(m_k)}\}$ to denote the set of eigenvalues of $A$ with their multiplicities. 

If $(m_1, m_2,\ldots, m_k)$ is an ordered multiplicity list of an $n\times n$ matrix $A\in \mathcal{S}(G)$, then this list in non-decreasing order provides an integer partition of $n$ and is called a \emph{multiplicity partition} of $G$. 

An interesting question is to find the minimum length among all multiplicity partitions of $G$. This parameter is called the \emph{minimum  number of distinct eigenvalues of $G$} and is denoted by $q(G)$. 
For a given $n\times n$ matrix $A$, we use $q(A)$ to denote the number of distinct eigenvalues of $A$. Thus, for a graph $G$ we have 
\[q(G)=\min\{q(A)\,:\, A\in \mathcal{S}(G)\}.\]

Throughout the paper, we refer to graphs $G$ with $q(G) = k$ as graphs with (or that allow) $k$ distinct eigenvalues.  The only graphs with one distinct eigenvalue are graphs with no edges, as the corresponding matrix has to be a scalar multiple of the identity matrix. Graphs with two distinct eigenvalues have proven to be a very interesting class of graphs. For instance, it is known that a nontrivial graph $G$ satisfies $q(G)=2$ if and only if there is an orthogonal matrix in $\mathcal{S}(G)$; see \cite{MR3118943}. Moreover, every connected (this can easily be extended to include disconnected graphs as well) graph $G$ is an induced subgraph of a graph with two distinct eigenvalues; see \cite[Theorem 5.2]{MR3118943}. These indicate the difficulty of characterizing graphs with two distinct eigenvalues. In this article, we study the edge density of nontrivial connected graphs that have 2 distinct eigenvalues. In terms of matrices, this is equivalent to the study of the possible patterns, or the number, of nonzero entries of symmetric orthogonal matrices. Various questions regarding the nonzero pattern of orthogonal matrices, not necessarily symmetric, have been studied, see for example \cite{MR1240965}, \cite{MR1648304}, and \cite{MR429637}.

We use standard notation for graphs and matrices. Cycles and paths on $n$ vertices are denoted by $C_n$ and $P_n$, respectively. We use $K_n$ to denote the complete graph on $n$ vertices and $K_{m,n}$ for the complete bipartite graph with partitions of orders $m$ and $n$. Let $G=(V(G),E(G))$ be a graph and $u,v\in V(G)$. If $u$ and $v$ are adjacent, we use $\{u,v\}$, or $uv$ for short, to denote the edge connecting $u$ and $v$.
The set of vertices that are adjacent to $v$ is denoted by $N(v)$, and $N[v]=\{v\}\cup N(v).$ The \emph{distance} between $u$ and $v$ in the graph $G$, denoted by $d_G(u,v)$, is the number of edges in a shortest path between $u$ and $v$ in $G$. The set of vertices of distance $i$ from $v$ is denoted by $N_i(v)$. Thus, $N_0(v)=\{v\}$ and $N_1(v)=N(v)$ for a vertex $v\in V(G).$ If $U \subseteq V(G)$, then we let $G[U]$ denote the subgraph of $G$ induced by the vertices in $U$.

A well-known lower bound for $q(G)$ is obtained by finding a largest unique shortest path between two vertices. This is explained in the following lemma. 

\begin{lemma}\label{}[Theorem 3.2, \cite{MR3118943}]
If vertices $x,y$ of a graph $G$ are at distance $d$, and if there is only one path of length $d$ between $x$ and $y$, then $q(G)\geq d+1$. 
\end{lemma}

\begin{corollary}\label{UniquePath}
Let $G$ be a graph $G$ with $q(G)=2$. If $G$ contains a path $P=xuy$ of length $2$, then either $x$ and $y$ are adjacent, or there is a path $P'=xvy$ in $G$ distinct from $P$.
\end{corollary}
\begin{corollary}\label{2Connected}
If $G$ is a connected graph on at least 3 vertices, and $q(G)=2$, then $G$ is 2-connected.
\end{corollary}
\begin{proof}
Using Corollary \ref{UniquePath}, every edge of $G$ lies in a cycle, and thus $G$ is $2$-connected. (This is also an immediate consequence of Corollary 4.3 in \cite{MR3118943}.)
\end{proof}

A set of vertices in a graph $G$ is called \emph{independent} if no two of them are adjacent. 

\begin{lemma}\label{maxindependnet}[Lemma 2.3, \cite{MR4044603}]
Let $G$ be a connected graph on $n$ vertices with $q(G)=2$. If $S$ is an independent set of vertices, then $|S|\leq \lfloor \frac{n}{2}\rfloor.$
\end{lemma}

\begin{lemma}\label{Lem:common-neighbors}[Theorem 4.4, Erratum to \cite{MR3118943}]
Let $G$ be a connected graph on $n\geq 3$ vertices with $q(G)=2$. If $\{u_1,u_2,\ldots,u_k\}$ is an independent set of vertices
such that for each $i = 1, 2,\ldots , k$, there exists a $j \neq i$ with $N(u_i) \cap N(u_j) \neq \emptyset$, 
then 
\[
\left|\bigcup_{i\neq j}\left(N(u_i)\cap N(u_j)\right)\right| \geq k.
\] 
\end{lemma}

Let $G$ be a graph and $v\in V(G)$. We use $\textrm{jdup}(G,v)$ to denote the graph obtained from $G$ by adding a new vertex $u$ and adding all edges between $u$ and $N[v]$, and we say $u$ is a \emph{joined duplicated vertex} of $v$. If the vertex $u$ is adjacent to all of the neighbors, $N(v)$, of $v$ and not adjacent to $v$, then the new graph is denoted by $\textrm{dup}(G,v)$ and we say $u$ is a \emph{duplicated vertex} of $v$. For a list $L$ of vertices of $G$, the graphs $\jdup(G,L)$, or $\dup(G,L)$, are the graphs that result from joined duplicating, or duplicating, each vertex in $L$ in sequence. 

\begin{lemma}(Corollary 3.3, \cite{Butler}) 
If $G$ has no isolated vertices, and a matrix $A\in \mathcal{S}(G)$ achieves the multiplicity list of $(m_1,\ldots,m_k)$, then for any $v\in V(G)$ there is a matrix $B \in \mathcal{S}(\textrm{jdup}(G,v))$ that achieves the multiplicity list $(m_1+1,\ldots,m_k)$.
\end{lemma}

\begin{corollary}\label{q of jdup}
If $G$ is a connected graph, then $q(\textrm{jdup}(G,v))\leq q(G)$.	
\end{corollary}

\section{Candle Constructions}\label{Constructions}

In this section, we introduce two families of connected graphs with $q$-value equal to $2$. In Section \ref{AllowsSection} we will show that these graphs are edge-minimal with respect to having $q$-value equal to $2$.

For $k\geq 3$, the \emph{double-ended candle} of diameter $k$ is the graph $G_k$ formed from $P_{k+1}$ (with vertices labeled $1, 2, 4,\ldots, 2k-2,2k$) by duplicating the vertices $2, 4,\ldots, 2k-2$ in sequence; that is, $G_k=\dup(P_{k+1},[2, 4,\ldots, 2k-2])$. We take $G_2=C_4$. The vertices of $G_k$ are labeled according to Figure \ref{DoubleCandle}. 
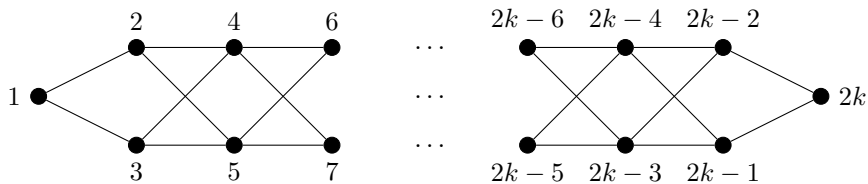
\begin{figure}[h!]
\centering
\begin{tikzpicture}[scale=0.65, vrtx/.style args = {#1/#2}{%
      circle, draw, fill=black, inner sep=0pt,
      minimum size=6pt, label=#1:#2}]
\node (0) [vrtx=left/1] at (0,0) {};
\node (1) [vrtx=above/2] at (2,1) {};
\node (2) [vrtx=above/4] at (4,1) {};
\node (3) [vrtx=above/6] at (6,1) {};
\node (k-3) [vrtx=above/$2k-6$] at (10,1) {};
\node (k-2) [vrtx=above/$2k-4$] at (12,1)  {};
\node (k-1) [vrtx=above/$2k-2$] at (14,1)  {};
\node (k) [vrtx=right/$2k$] at (16,0)  {};
\node (k+1) [vrtx=below/$2k-1$] at (14,-1) {};
\node (k+2) [vrtx=below/$2k-3$] at (12,-1)  {};
\node (k+3) [vrtx=below/$2k-5$] at (10,-1)  {};
\node (2k-1) [vrtx=below/$7$] at (6,-1) {};
\node (2k) [vrtx=below/$5$] at (4,-1)  {};
\node (2k+1) [vrtx=below/$3$] at (2,-1)  {};
\node (dotsu) at (8,1) {$\ldots$};
\node (dotsc) at (8,0) {$\ldots$};
\node (dotsl) at (8,-1) {$\ldots$};
\draw (0) edge (1);
\draw (1) edge (2);
\draw (2) edge (3);
\draw (k-3) edge (k-2);
\draw (k-2) edge (k-1);
\draw (k-1) edge (k);
\draw (k) edge (k+1);
\draw (k+1) edge (k+2);
\draw (k+2) edge (k+3);
\draw (2k-1) edge (2k);
\draw (2k) edge (2k+1);
\draw (2k+1) edge (0);
\draw (1) edge (2k);
\draw (2) edge (2k+1);
\draw (2) edge (2k-1);
\draw (3) edge (2k);
\draw (k-3) edge (k+2);
\draw (k-2) edge (k+3);
\draw (k-2) edge (k+1);
\draw (k-1) edge (k+2);
\end{tikzpicture}
\caption{Double-ended candle $G_k$.}\label{DoubleCandle}
\end{figure}

For $k\geq 2$, the \emph{single-ended candle} of diameter $k$ is the graph $G'_k$ formed from $G_k$ by joined duplicating the vertex $2k$ (see Figure \ref{DoubleCandle});  $G'_k=\jdup(G_k,2k)$. We take $G_1'=K_3$. The vertices of $G'_k$ are labeled according to Figure \ref{Single-Ended Candle}. 

\begin{figure}[h!]
\centering
\begin{tikzpicture}[scale=0.65, vrtx/.style args = {#1/#2}{%
      circle, draw, fill=black, inner sep=0pt,
      minimum size=6pt, label=#1:#2}]
\node (0) [vrtx=left/1] at (0,0) {};
\node (1) [vrtx=above/2] at (2,1) {};
\node (2) [vrtx=above/4] at (4,1) {};
\node (3) [vrtx=above/6] at (6,1) {};
\node (k-2) [vrtx=above/$2k-4$] at (10,1) {};
\node (k-1) [vrtx=above/$2k-2$] at (12,1)  {};
\node (k) [vrtx=above/$2k$] at (14,1)  {};
\node (k+1) [vrtx=below/$2k+1$] at (14,-1) {};
\node (k+2) [vrtx=below/$2k-1$] at (12,-1)  {};
\node (k+3) [vrtx=below/$2k-3$] at (10,-1)  {};
\node (2k-2) [vrtx=below/$7$] at (6,-1) {};
\node (2k-1) [vrtx=below/$5$] at (4,-1)  {};
\node (2k) [vrtx=below/$3$] at (2,-1)  {};
\node (dotsu) at (8,1) {$\ldots$};
\node (dotsc) at (8,0) {$\ldots$};
\node (dotsl) at (8,-1) {$\ldots$};
\draw (0) edge (1);
\draw (1) edge (2);
\draw (2) edge (3);
\draw (k-2) edge (k-1);
\draw (k-1) edge (k);
\draw (k) edge (k+1);
\draw (k+1) edge (k+2);
\draw (k+2) edge (k+3);
\draw (2k-2) edge (2k-1);
\draw (2k-1) edge (2k);
\draw (2k) edge (0);
\draw (1) edge (2k-1);
\draw (2) edge (2k-2);
\draw (2) edge (2k);
\draw (3) edge (2k-1);
\draw (k-2) edge (k+2);
\draw (k-1) edge (k+1);
\draw (k-1) edge (k+3);
\draw (k) edge (k+2);
\end{tikzpicture}
\caption{Single-ended candle $G'_k$.}\label{Single-Ended Candle}
\end{figure}
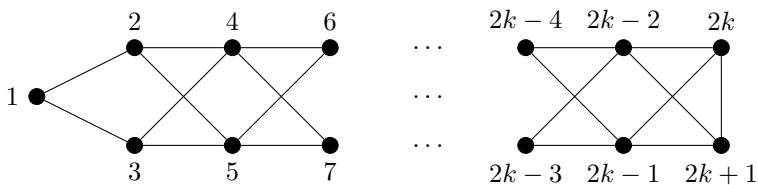

To find the number of edges in $G_k$, we simply halve the total sum of the vertex degrees. For example, when $k\geq 3$, each $G_k$ has two vertices of degree $2$, four vertices of degree $3$, and $2(k-3)$ vertices of degree $4$. 
\begin{remark}\label{CandleDensity}
The number of edges of $G_k$, and $G_k'$ can be expressed in terms of their number of vertices as follows.
\begin{itemize}
\item For $k\geq 2$, $|V(G_k)|=2k=n$ and $|E(G_k)|=4k-4=2n-4$.  
\item For $k\geq 1$, $|V(G'_k)|=2k+1=n$ and $|E(G'_k)|=4k-1=2n-3$.
\end{itemize}
\end{remark}
The following lemma gives a construction of orthogonal matrices for all candle graphs.
\begin{lemma}\label{qCandle}
For $k\geq 2$, we have $q(G_k)=q(G'_k)=2$. And $q(G_1')=2$.
\end{lemma}
\begin{proof}
The adjacency matrix of (the non-empty graph) $G_1'=K_3$ has two distinct eigenvalues, so $q(G_1')=2$.

The graphs $G_k$ and $G'_k$ with $k\geq 2$ are connected with at least four vertices, so they have at least two distinct eigenvalues. To see that they achieve two distinct eigenvalues, we construct orthogonal matrices for each family of graphs.  

Let 
\begin{eqnarray*}
R=\left[ \begin{array}{rr} 
-1 & 1 \\
1 & -1 \end{array} 
\right],\ \ 
J= \left[ \begin{array}{rr}
1 & 1 \\
1 & 1 \end{array} \right],\ \  
O= \left[ \begin{array}{rr}
0 & 0 \\
0 & 0 \end{array} \right],\ \ 
T_1=\sqrt{2} [1, -1],\ \  \mbox{and}\ \  T_2=\sqrt{2} [1, 1].
\end{eqnarray*}

Let $X=[X_{i,j}]$ be a block-partitioned matrix where $X_{i,j}$ denotes the block matrix in block row $i$ and block column $j$. We construct the matrix $A\in \mathcal{S}(G_k)$ using block matrices. We let the diagonal blocks be zero, and the blocks above the diagonal blocks be defined as described in what follows. The matrix is symmetric so the blocks below the diagonal blocks are simply the transposes of their corresponding matrices above the diagonal blocks. 

\noindent\emph{Case 1:} $n=2k$ and $k$ is odd:

\begin{equation*}
X_{i,j}=
    \begin{cases}
      J ,& \for (i,j)= (2,3),(4,5), \ldots,(k-3,k-2),(k-1,k) \ \\
      R,& \for (i,j)=(3,4),(5,6), \ldots,(k-2,k-1) \\
      T_1,& \for (i,j)=(1,2)\\
      T_1^T,& \for (i,j)=(k,k+1)\\
      O,& \textrm{otherwise}\\
    \end{cases}
\end{equation*}

\noindent\emph{Case 2:} $n=2k$ and $k$ is even:

\begin{equation*}
X_{i,j}=
    \begin{cases}
      J ,& \for (i,j)= (2,3),(4,5), \ldots,(k-4,k-3),(k-2,k-1) \ \\
      R,& \for (i,j)=(3,4),(5,6), \ldots,(k-1,k) \\
      T_1,& \for (i,j)=(1,2)\\
      T_2^T,& \for (i,j)=(k,k+1)\\
      O,& \textrm{otherwise.}\\
    \end{cases}
\end{equation*}

In each case, every row of $X_{i,j}$ has Euclidean length $2$.
Since  each of the $2 \times 3$, and $2 \times 4$ matrices 
$[T_1^T\; J]$, $[J\; R]$,   $[J \; T_1^T]$, and $[R \; T_2^T]$
have orthogonal rows, each pair of rows of $X$ coming from the same $2\times 2$ block 
are orthogonal.  In the first case, for $i\neq j$, the $(i,j)$-block of $X^2$ is one 
of $T_1J$, $JR$, or their transposes, or the zero block.  In the second case, for $i\neq j$, the $(i,j)$-block of $X^2$ is one 
of  $T_1J$, $JR$,  $T_2R$ or their transposes, or the zero block.
In each of the cases the matrix is the zero matrix.  
Hence $X^T X=X^2= 4I$ in either case.  We conclude that $\frac{1}{2} X$ is an orthogonal matrix whose graph 
is the double-ended candle on $n=2k$ vertices, so $q(G_k)=2$. Using Corollary \ref{q of jdup}, it follows that $q(G'_k)=2.$
\end{proof}

\begin{remark}
Each $G_k$ and $G'_k$ has independence number $k$. Thus, by Lemma 2.3 of \cite{MR4044603}, the only achievable multiplicity partition for two distinct eigenvalues is $[k,k]$ for $G_k$, and $[k,k+1]$ for $G'_k$.
\end{remark}

\begin{remark}
Graphs $G_k$ and $G_k'$ are obtained from $P_{k+1}$ by a sequence of vertex duplications (and joined duplications). The only graphs $G$ on $n$ vertices with $q(G)=n$ are the path graphs (see \cite{MR3118943}). So $q(P_{k+1})=k+1$, but after a sequence of vertex duplications we have $q(G_k)=q(G_k')=2$, a dramatic reduction in $q$-value.
\end{remark}

\section{The $q(G)=2$ Allows Problem}\label{AllowsSection}

The $q(G)=2$ allows problem asks: for a connected graph $G$ on $n$ vertices, how many edges are necessary to allow $q(G)=2$? The answer, along with the characterization of the extremal graphs, is given by the following theorem. We denote the graph on eight vertices given by the vertices and edges of the cube by $Q_3$.
\begin{theorem}\label{Characterization}
If $G$ is a connected graph on $n\geq 3$ vertices with $q(G)=2$, then
\[
|E(G)|\geq\begin{cases}
2n-4 ,& \text{if $n$ is even, and}\\
2n-3, & \text{if $n$ is odd.}
\end{cases}
\]
Moreover, the only graphs that meet this bound with $n$ even are $Q_3$ and the double-ended candles. The only graphs that meet this bound with $n$ odd are the single-ended candles.
\end{theorem}
Lemma \ref{Lower1} and Corollary \ref{Lower2} prove the inequalities in Theorem \ref{Characterization}. From Lemma \ref{qCandle}, we have $q(G_k)=q(G'_k)=2$. Corollary 6.9 from \cite{MR3118943} establishes $q(Q_3)=2$. Theorems \ref{EvenCharacterization} and \ref{OddCharacterization} show that these are the only extremal graphs. 

For a connected graph $G$ and $v\in V(G)$, the \emph{eccentricity of $v$}, denoted by $\epsilon(v)$, is the maximum distance of any vertex of $G$ from $v$. For $i=0,1,2,\ldots, \epsilon(v)$, let $d_i=|N_i(v)|,$ so $d_0=1$. Then $(d_0,d_1,\ldots,d_{\epsilon(v)})$ is called the \emph{distance sequence} of $v$ and $(d_0,d_1,\ldots,d_{\epsilon(v)-1})$ is called the \emph{truncated distance sequence} of $v$. A \emph{distance partition} of $G$ with respect to $v$ is the partition of $V(G)$ as follows: 
\begin{equation*}
    V(G)=\cup_{i=0}^{\epsilon(v)}N_{i}(v).
\end{equation*}
In the distance partition of $G$ with respect to $v$, if $u\in N_{i-1}(v)$ and $w\in N_{i}(v)$ are adjacent, then $u$ is called a \emph{predecessor} of $w$ and $w$ is called a \emph{successor} of $u$.  

\begin{lemma}\label{Lower1}
If $G$ is a connected graph on $n\geq 3$ vertices with $q(G)=2$, then $|E(G)|\geq 2n-4$.
\end{lemma}
\begin{proof}
Let $v$ be any vertex of $G$. If $u\in N_i(v)$ for $i\geq 2$, then $u$ is adjacent to at least two vertices in $N_{i-1}(v)$. This follows from Lemma \ref{UniquePath} for if $x$ is the unique neighbor of $u$ in $N_{i-1}(v)$, then there is some $y\in N_{i-2}(v)$ so that $uxy$ is the unique $uy$-path of length $2$.

Since there are at least $2|N_{i}(v)|$ edges between $N_i(v)$ and $N_{i-1}(v)$ for all $i\geq 2$, we have
\begin{equation}\label{lowerbound-edges}
|E(G)|\geq \deg(v)+2(n-\deg(v)-1)=2n-2-\deg(v).    
\end{equation}
To complete the proof note that $\delta(G)\geq 2$ by Corollary \ref{2Connected}, since $G$ is $2$-connected. Substituting $\deg(v)=2$ into the bound in \eqref{lowerbound-edges} gives $|E(G)|\geq 2n-4$. If $\delta(G)=3$, then substituting $\deg(v)=3$ into the bound \eqref{lowerbound-edges} gives $|E(G)|\geq 2n-5$. However, in this case the set $N_{\epsilon(v)}(v)$ either contains a vertex with three predecessors, or it contains at least two vertices joined by an edge. In either case, we have at least one more edge than accounted for above, and the bound is improved to $|E(G)|\geq 2n-4$. Finally, if $\delta(G)\geq 4$, then $|E(G)|\geq 2n$.
\end{proof}

The following lemmas are used in characterizing the extremal graphs $G$ with $q(G)=2$. 

\begin{lemma}\label{nondecreasingNjs}
Let $G$ be a connected graph on $n\geq 3$ vertices with $q(G)=2$ and $\delta(G)\geq 3$, and let $v\in V(G)$. If every vertex of $N_i(v)$ for $i=2,3,\ldots,\epsilon(v)-1$ has exactly two predecessors, and each $N_i(v)$ for $i=1, 2,\ldots,\epsilon(v)-1$ is an independent set, then the truncated distance sequence of $v$ is nondecreasing.
\end{lemma}
\begin{proof}
First note that $d_0=1\leq d_1.$ Suppose $d_{j-1}>d_{j}$ for some $j\in \{2,3,\ldots, \epsilon(v)-1\}.$ Since the number of edges between $N_{j-1}(v)$ and $N_{j}(v)$ is $2d_j$, some vertex $u\in N_{j-1}(v)$ has fewer than 2 neighbors in $N_{j}(v)$. Consider the following cases:
\vskip.3\baselineskip
\noindent \emph{Case 1:} $u$ has no neighbor in $N_{j}(v)$. Since $u$ has no neighbors in $N_{j-1}(v)$ and two neighbors in $N_{j-2}(v)$ (or only one neighbor in $N_{j-2}(v)$ if $j=2$), we have $\textrm{deg}(u) \leq 2,$ a contradiction. 
\vskip.3\baselineskip
\noindent \emph{Case 2:} $u$ has exactly one neighbor $w$ in $N_{j}(v)$. Because $\textrm{deg}(w)\geq 3$, $w$ has a neighbor $x$ in $N_{j+1}(v)$. Then $uwx$ is a unique $ux$-path of length 2 in $G$, which implies that $q(G)\geq 3,$ a contradiction.
\end{proof}

\begin{lemma} \label{noThree3s}
Let $G$ be a connected graph with $\delta(G)\geq 3$, and let $v\in V(G)$.  If every vertex of $N_{i}(v)$ for $i=2,3,\ldots, \epsilon(v)-1$ has exactly two predecessors, and each $N_{i}(v)$ is an independent set, and $|N_{j}(v)| = |N_{j-1}(v)| = |N_{j-2}(v)| = 3$ for some index $3 \leq j \leq \epsilon(v)-1$, then $q(G) \geq 3$.
\end{lemma}

\begin{proof}
Let $N_j(v) = \{a,b,c\}$ and $N_{j-1}(v) = \{x,y,z\}$.  If any vertex, say $x$, in $N_{j-1}(v)$ only has one successor, say $a$, in $N_j(v)$, then there exists a unique shortest path of length 2 joining $x$ to a successor of $a$.  That is, each predecessor of $\{a,b,c\}$ must have at least two neighbors in $\{a,b,c\}$ and no predecessor can have three neighbors in $\{a,b,c\}$, so $G[\{a,b,c,x,y,z\}]\cong C_6$.  By the same argument, if $N_{j-2}(v) = \{r,s,t\}$, then $G[\{x,y,z,r,s,t\}]\cong C_6$. Without loss of generality, suppose the edges between $\{a,b,c,x,y,z\}$ are $ax,xb,bz,zc,cy,ya$, and suppose that $r$ is the common predecessor of $x$ and $y$. Then $cyr$ is the unique shortest $cr$-path in $G$, so $q(G) > 2$.  
\end{proof}

\begin{lemma} \label{candles}
Let $G$ be a connected graph on $n\geq 3$ vertices with $q(G) = 2$ and $\delta(G) = 2$.  Let $v \in V(G)$ be a vertex of degree 2.  Suppose every vertex in $N_i(v)$ for $i=2, 3, \ldots \epsilon(v)$ has exactly two predecessors.

\begin{enumerate}
  \item[(a)] If $G$ contains an induced subgraph $H$ that is a double-ended candle, with ends $v$ and a vertex $u \in N_j(v)$ that has no successors, and each $N_i(v)$ is an independent set, then $G=H$.
  \item[(b)] If $G$ contains an induced subgraph $K$ that is a single-ended candle, with ends $v$ and a pair of adjacent vertices $\{w_1, w_2\} \subseteq N_j(v)$ that have no successors, and each $N_i(v)$ is an independent set except for the single edge $w_1w_2$, then $G = K$.
  \item[(c)] If $G$ contains an induced subgraph $J$ that is a single-ended candle, with ends $v$ and a pair of adjacent vertices $\{w_1, w_2\} \subseteq N_j(v)$, and each $N_i(v)$ is an independent set except for the single edge $w_1w_2$, then $G[\cup_{i=0}^j N_i(v))]=J$.
  \item[(d)] If $G$ contains an induced subgraph $L$ that is a double-ended candle, with ends $v$ and a vertex $u \in N_j(v)$ that has no successors, and each $N_i(v)$ is an independent set except for a single edge joining two vertices $\{w_1, w_2\} \subseteq V(L)\cap N_l(v)$ for some $1\leq l < \epsilon(v)$, then $G=L$.
\end{enumerate}
\end{lemma}

\proof We note the similarity of the four results and prove the latter three as modifications of the first.
\begin{enumerate}
  \item[(a)] Suppose $G \neq H$.  Denote $\{x_t, y_t\} = V(H) \cap N_{j-t}(v)$ for $1\leq t < j$.  Since $G$ is connected, let $s$ be the minimum value such that $1\leq s < j$  and at least one of $\{x_s, y_s\}$ has a neighbor $z$ in $V(G) \setminus V(H)$.  Since $x_s, y_s$ each already have the maximum number of predecessors, $z\in N_{j-(s-1)}$.  Let $a = x_{s-1}$ if $s \neq 1$, and let $a = u$ otherwise.  If $s\neq 1$, then by the minimality of $s$, any common successor of $a$ and $z$ must be in $H$, which is impossible.  If $s = 1$, then $a$ and $z$ cannot have a common successor since $u$ has no successors.  To prevent a unique $az$-path of length 2, $z$ must therefore have $x_s, y_s$ as its two predecessors.  Let $b = x_{s+1}$ if $s\neq j-1$, and let $b = v$ otherwise.  Then three independent vertices $a, z, b$ have only two pairwise common neighbors and, using Lemma \ref{Lem:common-neighbors}, $q(G) \geq 3$.
  
  \item[(b)] The result follows if we replace $H$ with $K$ and, without loss of generality, $u$ with $w_1$ in the proof of part (a).
  
  \item[(c)] In this case, we replace $H$ with $J$, $G$ with $G[\cup_{i=0}^j N_i(v))]$  and, without loss of generality, $u$ with $w_1$ in the proof of part (a).  The result then follows as in the proof of part (a) with one additional change:  if $s=1$, then $a=w_1$ and $z$ cannot have a common successor because any successor $z'$ of $w_1$ must also be a successor of $w_2$ to prevent a unique $z'w_2$-path of length 2; it follows that $z'$ must have more than 2 predecessors.
  
  \item[(d)] The result follows if we replace $H$ with $L$ in the proof of part (a). \qed
  
\end{enumerate}

We note that the $q$-values of the graphs obtained in case (d) of Lemma \ref{candles} are not known in general. However, when an edge is added to $G_3$ or $G_4$ as shown in Figure \ref{G3andG4withedge}, we have the corresponding matrices $M_1$ and $M_2$, respectively, with two distinct eigenvalues. In particular, $\sigma(M_1)=\{-3^{(3)}, 3^{(3)}\}$ and $\sigma(M_2)=\{-2\sqrt{2}^{(4)}, 2\sqrt{2}^{(4)}\}$, where

\[
M_1=\left[\begin{array}{rrrrrr}
-1 & 2 & 2 & 0 & 0 & 0 \\
2 & 0 & 1 & \sqrt{2} & -\sqrt{2} & 0 \\
2 & 1 & 0 & -\sqrt{2} & \sqrt{2} & 0 \\
0 & \sqrt{2} & -\sqrt{2} & 1 & 0 & 2 \\
0 & -\sqrt{2} & \sqrt{2} & 0 & 1 & 2 \\
0 & 0 & 0 & 2 & 2 & -1
\end{array}\right],
\]
and
\[
M_2=\left[\begin{array}{rrrrrrrr}
2 & \sqrt{2} & \sqrt{2} & 0 & 0 & 0 & 0 & 0 \\
\sqrt{2} & -2 & 0 & -1 & 1 & 0 & 0 & 0 \\
\sqrt{2} & 0 & -2 & 1 & -1 & 0 & 0 & 0 \\
0 & -1 & 1 & 0 & -2 & 1 & 1 & 0 \\
0 & 1 & -1 & -2 & 0 & 1 & 1 & 0 \\
0 & 0 & 0 & 1 & 1 & 2 & 0 & -\sqrt{2} \\
0 & 0 & 0 & 1 & 1 & 0 & 2 & \sqrt{2} \\
0 & 0 & 0 & 0 & 0 & -\sqrt{2} & \sqrt{2} & -2
\end{array}\right].
\]

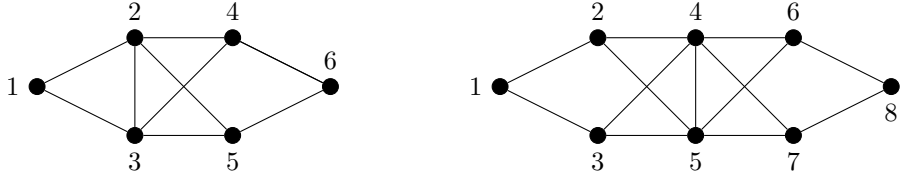
\begin{figure}[h!]
\centering
\begin{tikzpicture}[scale=0.65, vrtx/.style args = {#1/#2}{%
      circle, draw, fill=black, inner sep=0pt,
      minimum size=6pt, label=#1:#2}]
\node (0) [vrtx=left/1] at (0,0) {};
\node (1) [vrtx=above/2] at (2,1) {};
\node (2) [vrtx=above/4] at (4,1) {};
\node (3) [vrtx=above/6] at (6,0) {};
\node (4) [vrtx=below/5] at (4,-1) {};
\node (5) [vrtx=below/3] at (2,-1)  {};
\draw (0) edge (1);
\draw (1) edge (2);
\draw (2) edge (3);
\draw (3) edge (4);
\draw (4) edge (5);
\draw (5) edge (0);
\draw (2) edge (3);
\draw (1) edge (4);
\draw (1) edge (5);
\draw (2) edge (5);
\end{tikzpicture}
\qquad\qquad
\begin{tikzpicture}[scale=0.65, vrtx/.style args = {#1/#2}{%
      circle, draw, fill=black, inner sep=0pt,
      minimum size=6pt, label=#1:#2}]

\node (0) [vrtx=left/1] at (0,0) {};
\node (1) [vrtx=above/2] at (2,1) {};
\node (2) [vrtx=above/4] at (4,1) {};
\node (3) [vrtx=above/6] at (6,1) {};
\node (4) [vrtx=below/$8$] at (8,0) {};
\node (5) [vrtx=below/$7$] at (6,-1)  {};
\node (6) [vrtx=below/$5$] at (4,-1) {};
\node (7) [vrtx=below/$3$] at (2,-1)  {};

\draw (0) edge (1);
\draw (1) edge (2);
\draw (2) edge (3);
\draw (3) edge (4);
\draw (4) edge (5);
\draw (5) edge (6);
\draw (7) edge (7);
\draw (7) edge (0);
\draw (1) edge (6);
\draw (7) edge (2);
\draw (7) edge (6);
\draw (2) edge (6);
\draw (2) edge (5);
\draw (3) edge (6);
\end{tikzpicture}
\caption{$G_3$ and $G_4$ with an extra edge.}\label{G3andG4withedge}
\end{figure}

\vskip.3\baselineskip
 Theorem \ref{EvenCharacterization} gives a characterization of graphs $G$ with $q(G)=2$ and $|E(G)|= 2n-4$.
\begin{theorem}\label{EvenCharacterization}
If $G$ is a connected graph on an even number of vertices $n\geq 4$ with $q(G)=2$ and $|E(G)|=2n-4$, then either $G=Q_3$ or $G$ is a double-ended candle.
\end{theorem}
\begin{proof}
In order for a connected graph $G$ with $q(G)=2$ to have $2n-4$ edges, we must have $\delta(G)\in\{2,3\}$. We treat each case separately.
\vskip.3\baselineskip
\noindent\emph{Case 1:} $\delta(G)=2$.
\vskip.3\baselineskip
Let $v$ be a vertex of degree $2$ in $G$.  Since $G$ has exactly $2n-4$ edges, every vertex in $N_i(v)$ with $i\geq 2$ has exactly two predecessors and each $N_i(v)$ is an independent set in $G$.

Suppose for some $j\geq 2$, $u\in N_j(v)$ has no successors. Then $u$ has exactly two neighbors $x_1,y_1$, both of which lie in $N_{j-1}(v)$. By the assumption, both $x_1$ and $y_1$ have exactly two neighbors in $N_{j-2}(v)$ (unless $j-2=0$ in which case $x_1,y_1$ are both adjacent to $v$). Note that if $z\in N_{j-2}(v)$ is adjacent to $x_1$ but not to $y_1$, then $zx_1u$ is the unique $zu$-path of length $2$. Thus there are vertices $x_2,y_2\in N_{j-2}(v)$ so that $\{x_2,y_2\}=N(x_1)\cap N_{j-2}(v)=N(y_1)\cap N_{j-2}(v)$.  Again, from the assumption, either $x_2,y_2\in N_1(v)$, or they both have exactly two neighbors in $N_{j-3}(v)$. Applying the same argument, we see that there are $x_3,y_3\in N_{j-3}(v)$ so that $\{x_3,y_3\}=N(x_2)\cap N_{j-3}(v)=N(y_2)\cap N_{j-3}(v)$. Iterating this argument shows that there is $H\subseteq G$ where $H$ is a double-ended candle ($u$ and $v$ are the ends of $H$).  By Lemma \ref{candles} part (a), $G=H$.

\vskip.3\baselineskip
\noindent\emph{Case 2:} $\delta(G)=3$.
\vskip.3\baselineskip
Let $v$ be a vertex of degree $3$ in $G$. Substituting $\deg(v)=3$ into the inequality in \eqref{lowerbound-edges} gives $|E(G)|\geq 2n-5$. Since $|E(G)|=2n-4$ we conclude that either \emph{2(a)}: every vertex in $N_i(v)$ for $i\geq 2$ has exactly two predecessors except for one vertex that has exactly three predecessors, and every $N_i(v)$ is an independent set in $G$; or \emph{2(b)}: every vertex in $N_i(v)$ has exactly two predecessors, and $N_i(v)$ is an independent set for all $i\neq j$, while $N_j(v)$ has exactly one internal edge.
\vskip.3\baselineskip
\noindent\emph{Case 2(a):} There is exactly one vertex $u$ with three predecessors, and every $N_i(v)$ is independent.
\vskip.3\baselineskip
Every vertex besides $u$ has a successor so $N_{\epsilon(v)}(v)=\{u\}$ and $|N_{\epsilon(v)-1}(v)|=3.$ Also $\textrm{deg}(v)=3,$ so $|N_{0}(v)|=1,$ and $|N_{1}(v)|=3.$ Since by Lemma \ref{nondecreasingNjs} the truncated distance sequence of $v$ is nondecreasing, and by Lemma \ref{noThree3s} we cannot have three consecutive distance sets of size $3$, the only possible distance sequences are $(1,3,1)$ and $(1,3,3,1)$. In the first case, $G\cong K_{2,3}$. From Lemma \ref{Lem:common-neighbors} we see $q(K_{2,3})>2$, giving a contradiction. In the second case, $G[N_{1}(v)\cup N_{2}(v)]\cong C_6$ and it follows that $G \cong Q_3.$

\vskip.3\baselineskip
\noindent\emph{Case 2(b):} Every vertex in $N_i(v)$ with $i\geq 2$ has exactly two predecessors, and $N_i(v)$ is an independent set for all $i\neq j$, while $N_j(v)$ contains exactly one edge.
\vskip.3\baselineskip
Let $w_1w_2$ be the edge in $N_j(v)$. Since $\delta(G)=3$ and every vertex in $N_i(v)$ for $i\geq 2$ has exactly two predecessors, there is no vertex $z\notin \{w_1,w_2\}$ with no successors. It follows that $\epsilon(v)=j$, $N_j(v)=\{w_1,w_2\}$, and since $\deg(v)=3$ we have $j\geq 2$. If either $w_1$ or $w_2$ has a neighbor that is not a neighbor of the other, we have a unique shortest path of length $2$, which gives a contradiction. Thus there are vertices $x,y\in N_{j-1}(v)$ so that $w_1,w_2$ are both adjacent to both of $x$ and $y$. 
In fact, since $w_1$ and $w_2$ have used the maximum number of predecessors, we have $N_{j-1}(v)=\{x,y\}.$ Now $|N_1(v)|=3$ and $|N_{j-1}(v)|=2$ with $j\geq 3$, which contradicts Lemma \ref{nondecreasingNjs}.
\end{proof}

By Lemma \ref{Lower1}, a connected graph $G$ on $n\geq 3$ vertices with $q(G)=2$ has at least $2n-4$ edges. Theorem \ref{EvenCharacterization} gives a characterization of graphs $G$ with $q(G)=2$ and $|E(G)|=2n-4$. Therefore, a graph $G$ on $n\geq 3$ vertices with $q(G)=2$ that is not listed in Theorem \ref{EvenCharacterization} has at least $2n-3$ edges, in particular graphs on odd number of vertices. Theorem \ref{OddCharacterization} gives a characterization of graphs $G$ with $q(G)=2$ and $|E(G)|= 2n-3$. 
 
\begin{corollary}\label{Lower2}
If $G$ is a connected graph on an odd number of vertices $n\geq 3$ with $q(G)=2$, then $|E(G)|\geq 2n-3$.
\end{corollary}

\begin{lemma}\label{ind-distance-part} Let $G$ be a connected graph on an odd number of vertices $n \ge 3$ with $q(G)=2$. If for a vertex $v\in V(G)$ each distance set $N_i(v)$ for $i=1, \dots, \epsilon(v)$ is an independent set, then $q(G) \ge 3$. 
\end{lemma}
\begin{proof}
 Let $\displaystyle {A= \cup_{i \ odd} \, N_i(v)}$ and $\displaystyle B = \cup_{i \ even}\,  N_i(v)$.  Then $A \cup B = V(G)$ and $A \cap B = \emptyset$.  Since the sets $N_i(v) $ are independent, so are $A$ and $B$.  Since $|A| + |B| = n$, we have $|A| \neq |B|$.  Then either $|A| > n/2$ or $|B| > n/2$.  By Lemma \ref{maxindependnet}, $q(G) \geq 3$.
\end{proof}

\begin{theorem}\label{OddCharacterization}
If $G$ is a connected graph on an odd number of vertices $n\geq 3$ with $q(G)=2$ and $|E(G)|=2n-3$, then $G$ is a single-ended candle.
\end{theorem}
\begin{proof}
Suppose $G$ is a connected graph with an odd number of vertices, $q(G)=2$, and $|E(G)|=2n-3$. This implies that $\delta(G)\in\{2,3\}$.
\vskip.3\baselineskip
\noindent\emph{Case 1:} $\delta(G)=2$.
\vskip.3\baselineskip

Let $v$ be a vertex of degree $2$. Then either \emph{1(a):} every vertex in $N_i(v)$ for $i\geq 2$ has exactly two predecessors except for one vertex that has three predecessors, and each $N_i(v)$ is independent; or \emph{1(b):} every vertex in $N_i(v)$ has exactly two predecessors, and every $N_i(v)$ with $i\neq j$ is independent while $N_j(v)$ contains exactly one edge.

\vskip.3\baselineskip
\noindent\emph{Case 1(a):} By Lemma \ref{ind-distance-part}, $q(G)\geq 3$. 
\vskip.3\baselineskip

\noindent\emph{Case 1(b):} 
First we note that if $j=1$, then $N_1(v)=N(v)=\{w_1,w_2\}$ and $G[\{v\}\cup N_1(v)]$ is isomorphic to $K_3$, a single-ended candle. Now suppose $j\geq 2,$ and let $w_1w_2$ be the edge in $N_j(v)$.  We repeat the construction in Case 1 of Theorem \ref{EvenCharacterization} with a slight modification:  replace $u$ with $w_1$ and note that the two predecessors of $w_1$ must also be predecessors of $w_2$, to avoid a unique shortest path of length 2 between either predecessor and $w_2$.  The construction yields a single-ended candle $H$ starting at $v$, and ending at $\{w_1,w_2\}$.

If $j = \epsilon(v)$, then $H=G$ by Lemma \ref{candles} part (b).  Suppose $j<\epsilon(v)$.  We know $H = G[\cup_{i=0}^j N_i(v))]$ by Lemma \ref{candles} part (c).  In particular, $N_j(v)=\{w_1,w_2\}$. 
Choose an arbitrary vertex $v'\neq v$ with $\deg(v')=2$ (any vertex $v'\in N_{\epsilon(v)}(v)$ has degree $2$). Since $H = G[\cup_{i=0}^j N_i(v))]$, we have $d_G(v,v')>j$. Construct the distance partition of $G$ with respect to $v'$.

Vertices $w_1$ and $w_2$ are adjacent, so their distances from $v'$ differ by at most $1$. Since $N_j(v)=\{w_1,w_2\}$, every $y\in N_{j+1}(v)$ is adjacent to both $w_1$ and $w_2$. Let $y\in N_{j+1}(v)$ lie on a shortest $vv'$-path $P$. Suppose $w_1\in P$, and let $P'$ be a shortest $vw_2$-path. Then $vP'w_2yPv'$ is a shortest $vv'$-path. Thus $w_1$ and $w_2$ are at the same distance $k$ from $v'$.

Therefore in the distance partition of $G$ from $v'$, $w_1$ and $w_2$ lie in the same distance set, and are connected by an edge. Following the same argument from the start of the proof, we conclude that we can construct a single-ended candle $H'$ by working back towards $v'$ from $\{w_1,w_2\}$ in the same way we constructed $H$. Moreover, $H'=G[\cup_{i=0}^k N_i(v')]$. Any $vv'$-path must intersect $\{w_1, w_2\} = N_j(v)$.  If there exists $z\in V(H)\cap V(H') \setminus \{w_1, w_2\}$, then $d_G(z, v) < j$ and $d_G(z, v') < k$, and there exist $zv$- and $zv'$-paths that do not intersect $\{w_1, w_2\}$.  But then there is a $vv'$-path through $z$ that does not intersect $\{w_1, w_2\}$, a contradiction.  So $V(H)\cap V(H')=G[\{w_1,w_2\}]$ and $H\cup H'$ is a double-ended candle with an additional edge joining $w_1$ and $w_2$.  By Lemma \ref{candles} part (d), $H\cup H' = G$.  This is impossible from our assumptions, as $G$ has an odd number of vertices. 

\vskip.3\baselineskip
\noindent\emph{Case 2:} $\delta(G)=3$. 
\vskip.3\baselineskip

Let $v$ be a vertex of degree $3$. Now the bound in \eqref{lowerbound-edges} becomes $|E(G)|\geq 2n-5$. So we have two additional edges to account for. These edges can appear as \emph{2(a)}: one vertex with four predecessors; \emph{2(b)}: two vertices with three predecessors; \emph{2(c)}: one vertex with three predecessors and one edge internal to some $N_j(v)$; \emph{2(d)}: two edges internal to some $N_j(v)$ and $N_t(v)$ that share no endpoints; or \emph{2(e)}: two edges internal to some $N_j(v)$ that share an endpoint.
\vskip.3\baselineskip
\noindent\emph{Cases 2(a) and 2(b):} In these cases $q(G)\geq 3$ by Lemma \ref{ind-distance-part}. 
\vskip.3\baselineskip
\noindent\emph{Case 2(c):} There is exactly one vertex $u$ with three predecessors, and one edge $xy$ internal to some $N_j(v)$. \vskip.3\baselineskip
Note that if $u=x$, then all three predecessors of $x$ must also be adjacent to $y$, contradicting that exactly one vertex has three predecessors.  So $u\neq x$ and, similarly, $u\neq y$.  Then $x$ and $y$ each have exactly two predecessors or are in $N_1(v)$.  Following our logic from earlier, we see that as we trace back from $\{x,y\}$ to $v$, we obtain a single-ended candle $H \subseteq G$.  By the same argument that prevented $u$ from being $x$ or $y$, it is impossible for $u$ to be any vertex in $H$.

We claim that it is impossible for $u$ to be the unique vertex in $G$ with three predecessors.  Let $\{a_1, b_1, c_1\} = N_1(v)$, and assume, without loss of generality, that $\{a_1, b_1\} = V(H) \cap N_1(v)$.  Then $c_1$ must have two successors as $\delta(G) = 3$, which also forces $\{a_1, b_1\}$ to have those successors.  Consider $\{a_2, b_2, c_2, d_2\} \subseteq N_2(v)$ where $\{a_2, b_2\} = V(H) \cap N_2(v)$ are successors of $\{a_1, b_1\}$, and $\{c_2, d_2\}$, disjoint from $\{a_2, b_2\}$, are successors of $c_1$.  Since $c_2$ and $d_2$ cannot both have three predecessors, we may assume, without loss of generality, that $c_2$ has exactly two predecessors, one of which, say $b_1$, must be in $H$.  Both $b_2$ and $c_2$ already have their maximum number of predecessors (with $a_1$ not a predecessor of $c_2$ and $c_1$ not a predecessor of $b_2$), so to prevent the unique shortest path $b_2b_1c_2$ of length 2, $b_2$ and $c_2$ must share a common successor $c_3 \in N_3(v)$.  This creates paths of length 2 between $c_3$ and both $a_1$ and $c_1$.  Again, since $b_2$ and $c_2$ already have the maximum number of predecessors, this forces $c_3$ to have three predecessors, so that $u = c_3$.  Moreover, since $a_2$ already has the maximum number of predecessors, the third predecessor of $c_3$ cannot be $a_2$.

If $(a_2, b_2) = (x,y)$, then $c_3$ must be adjacent to $a_2$, which is a contradiction.  So $(a_2, b_2) \neq (x,y)$, in which case $\delta(G) = 3$ implies they must have successors.  Suppose $\{a_2, b_2\}$ have successors $\{a_3, b_3\} = V(H) \cap N_3(v)$.  Then $b_3b_2c_3$ is a path of length 2, and we repeat our argument from the previous paragraph.  Both $b_3$ and $c_3$ have the maximum number of predecessors (with $a_2$ not a predecessor of $c_3$ and $c_2$ not a predecessor of $b_3$), so to prevent the unique shortest path $b_3b_2c_3$ of length 2, $b_3$ and $c_3$ must share a common successor $c_4 \in N_4(v)$.  This creates paths of length 2 between $c_4$ and both $a_2$ and $c_2$.  As before, since $b_3$ and $c_3$ have the maximum number of predecessors, this forces $c_4$ to have a third predecessor, contradicting the uniqueness of $u$.

\vskip.3\baselineskip
\noindent\emph{Case 2(d):} There are two edges internal to some $N_j(v)$ and $N_t(v)$ that share no endpoints (here we may have $j=t$ or $j\neq t$). 
\vskip.3\baselineskip

Denote the edges by $uw$ and $xy$, and assume $\{x,y\} \subseteq N_j(v)$ and $\{u,w\} \subseteq N_t(v)$ where $t \leq j$. Suppose either $t \geq 2$ or $t=1$ and $j\geq 3$. Since every vertex in $N_i(v)$ for $i\geq 2$ has exactly $2$ predecessors and $t\leq j$, if $t\geq 2$, then tracing back from $\{u,w\}$ to $v$ as in our previous such construction, we obtain a single-ended candle $H\subseteq G$.   If $t=1$ and $j\geq 3$, tracing back from $\{x,y\}$, we obtain a single ended candle $H\subseteq G$ as well (it is possible that the edge $uw$ is contained in the first level of $H$ or not; it does not matter for the argument). 

Let $N_1(v) = \{a_1, b_1, c_1\}$ and without loss of generality assume that $\{a_1, b_1\} = V(H)\cap N_1(v)$. As all vertices in $H$ have the maximum number of predecessors and $\delta(G) =3$, there must be at least one vertex not in $H$ that is a successor of $c_1$, call it $c_2$.  

Consider $\{a_2, b_2, c_2\} \subseteq N_2(v)$ where $\{a_2, b_2\} = V(H) \cap N_2(v)$.  Following the logic of Case 2(c), $c_2$ must have another predecessor other than $c_1$, and we may assume that $c_2$ is adjacent to $b_1$ (we note again here that, for the remainder of this paragraph, $b_1\in\{u,w\}$ is possible but does not affect the argument). Note that $c_2$ is not adjacent to $b_2$ because otherwise we have $t=j=2$, which is impossible since $uw$ and $xy$ share no endpoints, or $t=1$ and $j=2$. Hence, $b_2b_1c_2$ is a path of length $2$ which forces the existence of a vertex $c_3 \in N_3(v)$, where $c_3$ may be one of $x$ or $y$. But now this vertex must  have at least 3 predecessors, or else $c_1c_2c_3$ is a unique shortest path between $c_1$ and $c_3$, a contradiction.

The only remaining case is if $t=1$ and $j=2$.  Since $\delta(G)=3$, $\{x,y\}$ are the only possibilities for vertices that have no successors.  So $N_{\epsilon(v)}(v) = N_2(v) = \{x,y\}$, and $|V(G)| = 6$, contradicting that $G$ has an odd number of vertices.

\vskip.3\baselineskip
\noindent\emph{Case 2(e):} There are edges $xy$ and $yz$ both contained in some $N_j(v)$.
\vskip.3\baselineskip
In this case, $x, y,$ and $z$ are the only vertices with no successors, $j=\epsilon(v)$, and $N_j(v)=\{x,y,z\}$.   Note that $N_{j-1}(v)$ has at most 3 vertices as otherwise we have a predecessor of one of $x, y, z$ that is not a  predecessor of the others which gives an immediate contradiction.  But because $\deg(v)= 3$, by Lemma \ref{nondecreasingNjs} all distance sets have size at least $3$. Taking $N_{j-1}(v) =\{x^\prime, y^\prime, z^\prime\}$, the bipartite graph given by the edges between $\{x,y,z\}$ and $\{x',y',z'\}$ is isomorphic to $C_6$.  If $x^\prime$ is adjacent to $x$ and $y$, then $x^\prime y z$ is the unique shortest path of length 2 joining $x^\prime$ to $z$.
\end{proof}
This completes the proof of Theorem \ref{Characterization}.

\section{Combinatorial Orthogonality}

Our characterization of graphs with minimum edge density and two distinct eigenvalues in Theorem \ref{Characterization} is related to earlier work on combinatorial orthogonality and quadrangular graphs in \cite{MR1240965}, \cite{MR1648304}, and \cite{MR429637}. The existence of a combinatorially orthogonal matrix $A\in \mathcal{S}(G)$ can be viewed as a relaxation of the property $q(G)=2$, and so our characterization in Theorem \ref{Characterization} can be seen as a refinement of a similar characterization in  \cite{MR429637}.

Two vectors $x,y\in\mathbb{R}^n$ are \emph{combinatorially orthogonal} if $|\supp(x)\cap\supp(y)|\neq 1$, where $\supp(x)=\{i\,:\, x_i\neq 0\}$ is the \emph{support} of $x$. A matrix $A$ is \emph{combinatorially orthogonal} if each of its rows are pairwise combinatorially orthogonal, and each of its columns are pairwise combinatorially orthogonal. It is easy to see that orthogonal vectors $x$ and $y$ are combinatorially orthogonal, and that if $A$ is an orthogonal matrix, then $A$ is combinatorially orthogonal. Thus for graphs $G$ with $q(G)=2$, there is a combinatorially orthogonal matrix $A\in\mathcal{S}(G)$. The converse does not hold. For example, $A(K_{2,3})\in\mathcal{S}(K_{2,3})$ is combinatorially orthogonal, but $q(K_{2,3})=3$.

Unlike graphs with $q$-value equal to $2$, graphs $G$ for which $\mathcal{S}(G)$ contains a combinatorially orthogonal matrix are easily characterized.
\begin{lemma}\label{CombOrthThm}
If $G$ is a graph, then $\mathcal{S}(G)$ contains a combinatorially orthogonal matrix if and only if the matrix $A(G)+I$ is combinatorially orthogonal.
\end{lemma}

\begin{proof}
If the matrix $A(G)+I$ is combinatorially orthogonal, then $\mathcal{S}(G)$ contains the combinatorially orthogonal matrix $A(G)+I$. Now suppose $\mathcal{S}(G)$ contains a combinatorially orthogonal matrix $M$. We show that $B=A(G)+I$ is combinatorially orthogonal. Consider two columns of $B$ corresponding to distinct vertices $u$ and $v$. If $uv\in E(G)$, then the $2\times 2$ submatrix of $B$ lying on rows and columns corresponding to $u$ and $v$ is the $2\times 2$ all-ones matrix. Therefore, the columns of $B$ corresponding to $u$ and $v$ are combinatorially orthogonal. If $uv\notin E(G)$, then $|N(u)\cap N(v)|\neq 1$. Otherwise, if $N(u)\cap N(v)=\{w\}$, then the columns of $M$ corresponding to $u$ and $v$ share only one nonzero entry which contradicts the fact that $M$ is combinatorially orthogonal. Thus $B$ is combinatorially orthogonal.
\end{proof}
Note that Lemma \ref{CombOrthThm} can be rephrased in terms of the structure of $G$. There is a combinatorially orthogonal matrix in $\mathcal{S}(G)$ if and only if every path of length $2$ in $G$ is contained in either a $C_3$ or a $C_4$ (the implication for graphs $G$ with $q(G)=2$ is given in Corollary \ref{UniquePath}).

In \cite{MR429637}, Reid and Thomassen state that a graph $G$ has property $p(r,s)$ if $G$ has a path of length $r$, and every path of length $r$ in $G$ is contained in a cycle of length $s$ (here $s\geq 3$ as our graphs are simple). Using this terminology, the previous paragraph establishes that there is a combinatorially orthogonal matrix in $\mathcal{S}(G)$ if and only if $G$ has property $p(2,\leq 4)$. Reid and Thomassen prove that if $G$ is a graph on $n$ vertices with property $p(2, \leq 4)$, then $|E(G)|\geq 2n-4$ (Theorem 4.1 in \cite{MR429637}). Moreover they show that if $G$ has property $p(2,4)$, then $|E(G)|\geq 2n-4$, and the bound is sharp (Corollary 4.2 in \cite{MR429637}). This bound is exactly the same as our bound in Lemma \ref{Lower1}.

The following terminology is taken from \cite{MR1648304}. A vertex $x$ of a graph $G$ is \emph{condensable} if $N(x)=\{u,v\}$, $N(u)=N(v)$, and $\deg(u)\geq 3$. A graph $H$ is a \emph{condensing} of $G$ if $H=G-x$ for some condensable vertex $x$ in $G$. And $H$ is a \emph{condensation} of $G$ if there is some sequence $H=G_1,G_2,\ldots,G_k=G$ so that for each $1\leq i\leq k-1$, $G_i$ is a condensing of $G_{i+1}$.  Reid and Thomassen prove that if $G$ is a connected graph with property $p(2,4)$ and $|E(G)|=2n-4$, then either $G=Q_3$, or $C_4$ is a condensation of $G$ (this is Theorem 4.5 in \cite{MR429637}, re-stated as Theorem 3.1 in \cite{MR1648304}). Note that all graphs $G$ for which $C_4$ is a condensation of $G$ have property $p(2,4)$.

In Theorem \ref{Characterization} we characterize the connected graphs $G$ with $q(G)=2$ and $|E(G)|=2n-4$ to be $Q_3$ together with the double-ended candles. It is unsurprising to note that $C_4$ is a condensation of each $G_k$. We have noted that $K_{2,3}$ has property $p(2,4)$, but $q(K_{2,3})=3$. This implies that not all graphs $G$ that have $C_4$ as a condensation have $q(G)=2$. Also note that there are connected graphs $G$ with $q(G)=2$ that do not have property $p(2,4)$. The single-ended candles $G'_k$ are such examples. In $G'_k$ the path $2k,2k+1,2k-1$ is contained in a $C_3$ but not in a $C_4$. From Theorem \ref{Characterization}, and Theorem 4.5 in \cite{MR429637}, we have the following immediate corollary.
\begin{corollary}\label{C4Characterization}
If $C_4$ is a condensation of $G$, and $q(G)=2$, then $G$ is a double-ended candle. 
\end{corollary}

In \cite{MR1648304}, Gibson and Zhang study combinatorially orthogonal matrices and quadrangular graphs. A graph is \emph{quadrangular} if $|N(x)\cap N(y)|\neq 1$ for all distinct $x,y$. Note that connected quadrangular graphs are exactly the graphs with property $p(2,4)$. The double-ended candles $G_k$ are quadrangular with $q(G_k)=2$. They are also bipartite. Gibson and Zhang extend the results from \cite{MR429637} above by showing that if $G\notin\{K_1,K_4\}$ is a connected non-bipartite quadrangular graph on $n$ vertices, then $|E(G)|\geq 2n-1$. Moreover $G$ meets this bound if and only if $G=K_6-C_4$ or $K_5-e$ is a condensation of $G$ (Theorem 4.1 in \cite{MR1648304}). From \cite{MR3904092} or \cite{Butler} we see $q(K_6-C_4)=2$. It would be interesting to determine the graphs $G$ with $q(G)=2$ for which $K_5-e$ is a condensation of $G$.

Finally, we note a connection to a result of Beasley, Brualdi, and Shader \cite{MR1240965}. An $n\times n$ matrix $A$ is \emph{partly decomposable} if there exist permutation matrices $P_1,P_2$ so that
\[
P_1AP_2=\left[\begin{array}{cc}
A_1 & 0\\
X & A_2
\end{array}\right]
\]
where the $A_i$ are square nonzero matrices. We say $A$ is \emph{fully indecomposable} if $A$ is not partly decomposable. Note that for matrices $A\in\mathcal{S}(G)$, $A$ is fully indecomposable if and only if $G$ is connected. In \cite{MR1240965}, it is shown that a fully indecomposable combinatorially orthogonal $n\times n$ matrix with $n\geq 2$ must have at least $4n-4$ nonzero entries, and the matrices that meet this bound are characterized (Theorem 2.2 in \cite{MR1240965}). This theorem applies to all matrices, not just symmetric matrices. However, if we subtract $n$ from their bound to account for the diagonal, and divide by $2$, we infer that if $A\in\mathcal{S}(G)$ for a connected graph $G$, then $|E(G)|\geq (3/2)n-2$. Comparing to the bound in Theorem \ref{Characterization}, we observe that no connected graph $G$ with $q(G)=2$ meets this bound.

\section{Conclusion and Future Directions}

We discuss what Theorem \ref{Characterization} tells us about determining which graphs $G$ have $q(G)=2$. We illustrate by considering graphs of order $5$ and $6$.

For graphs on $5$ vertices, the single-ended candle $G_2'$ is the unique connected graph with $7$ edges that has two distinct eigenvalues. So any other graph $G$ on $5$ vertices with at most $7$ edges must have $q(G)>2$. It remains to consider the graphs with at least $8$ edges. These graphs have complements with at most $2$ edges, and it is easy to see that each is isomorphic to a supergraph of $G_2'$. Orthogonal matrices in $\mathcal{S}(G_k)$ were constructed in Section \ref{Constructions}, and orthogonal matrices in $\mathcal{S}(G'_k)$ can be constructed similarly. For $n=5$ we have
\[
M=\left[\begin{array}{rrrrr}
0 & \sqrt{2} & \sqrt{2} & 0 & 0\\
\sqrt{2} & 0 & 0 & -1 & 1\\
\sqrt{2} & 0 & 0 & 1 & -1\\
0 & -1 & 1 & 1 & 1\\
0 & 1 & -1 & 1 & 1
\end{array}\right]\in\mathcal{S}(G_2').
\]
We see that $M$ has spectrum $\{(-2)^{(2)},2^{(3)}\}$, and a routine calculation verifies that $M$ has the Strong Spectral Property (see \cite{MR3665573} pg 10--11 for definitions). It follows from Theorem 10 in \cite{MR3665573} that for every supergraph $G$ of $G_2'$, there is a matrix in $\mathcal{S}(G)$ with spectrum $\{(-2)^{(2)},2^{(3)}\}$. Thus $q(G)=2$ for every supergraph $G$ of $G_2'$. These observations can be summarized as follows.
\begin{corollary}\label{5Vertices}
If $G$ be a connected graph on $5$ vertices, then $q(G)=2$ if and only if $G$ is a supergraph of $G_2'$.
\end{corollary}
This is not a new result, but Theorem \ref{Characterization} and the Strong Spectral Property give a simple way to state and understand it.

We now consider graphs on $6$ vertices and again our results are not new. The $q$-values of all connected graphs on $6$ vertices were calculated in \cite{MR3904092}. However, viewing these results in the context of Theorem \ref{Characterization} gives a new perspective. For graphs on $6$ vertices, the double-ended candle $G_3$ is the unique connected graph with $8$ edges that has two distinct eigenvalues. Three non-isomorphic graphs can be obtained by adding an edge to $G_3$; they are given in Figure \ref{6VtxFig}.
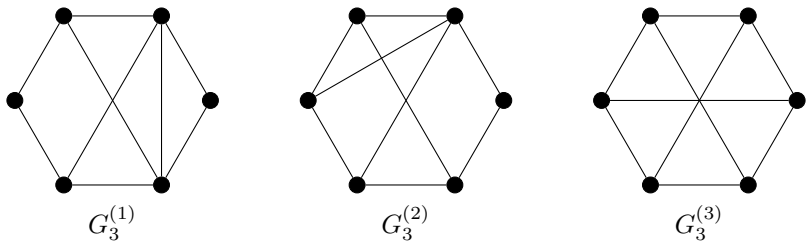
\begin{figure}[h!]
\centering
\begin{tikzpicture}[scale=0.65, vrtx/.style args = {#1/#2}{%
      circle, draw, fill=black, inner sep=0pt,
      minimum size=6pt, label=#1:#2}]
\node (0) [vrtx=left/] at (2,0) {};
\node (1) [vrtx=left/] at (1,1.73) {};
\node (2) [vrtx=left/] at (-1,1.73) {};
\node (3) [vrtx=left/] at (-2,0) {};
\node (4) [vrtx=left/] at (-1,-1.73) {};
\node (5) [vrtx=left/] at (1,-1.73) {};
\node (l1) [] at (0,-2.5) {$G_3^{(1)}$};
\draw (0) edge (1);
\draw (2) edge (1);
\draw (2) edge (3);
\draw (4) edge (3);
\draw (4) edge (5);
\draw (0) edge (5);
\draw (2) edge (5);
\draw (4) edge (1);
\draw (1) edge (5);
\node (01) [vrtx=left/] at (2+6,0) {};
\node (11) [vrtx=left/] at (1+6,1.73) {};
\node (21) [vrtx=left/] at (-1+6,1.73) {};
\node (31) [vrtx=left/] at (-2+6,0) {};
\node (41) [vrtx=left/] at (-1+6,-1.73) {};
\node (51) [vrtx=left/] at (1+6,-1.73) {};
\node (l2) [] at (0+6,-2.5) {$G_3^{(2)}$};
\draw (01) edge (11);
\draw (21) edge (11);
\draw (21) edge (31);
\draw (41) edge (31);
\draw (41) edge (51);
\draw (01) edge (51);
\draw (21) edge (51);
\draw (41) edge (11);
\draw (11) edge (31);
\node (02) [vrtx=left/] at (2+12,0) {};
\node (12) [vrtx=left/] at (1+12,1.73) {};
\node (22) [vrtx=left/] at (-1+12,1.73) {};
\node (32) [vrtx=left/] at (-2+12,0) {};
\node (42) [vrtx=left/] at (-1+12,-1.73) {};
\node (52) [vrtx=left/] at (1+12,-1.73) {};
\node (l3) [] at (0+12,-2.5) {$G_3^{(3)}$};
\draw (02) edge (12);
\draw (22) edge (12);
\draw (22) edge (32);
\draw (42) edge (32);
\draw (42) edge (52);
\draw (02) edge (52);
\draw (22) edge (52);
\draw (42) edge (12);
\draw (32) edge (02);
\end{tikzpicture}
\caption{Graphs $G_3^{(1)}$, $G_3^{(2)}$, and $G_3^{(3)}$.}\label{6VtxFig}
\end{figure}
Note that $G_3^{(2)}$ contains a pair of vertices at distance $2$ joined by a unique shortest path, which forces $q(G_3^{(2)})>2$. It follows that no matrix in $\mathcal{S}(G_3)$ has exactly two distinct eigenvalues and the Strong Spectral Property. So we are unable to replicate the simplicity of Corollary \ref{5Vertices}. However, from Theorem \ref{Characterization} we can conclude that of the $60$ connected graphs on $6$ vertices with at most $8$ edges, $G_3$ is the only graph $G$ with $q(G)=2$. It remains to determine which of the $52$ connected graphs on $6$ vertices with at least $9$ edges have two distinct eigenvalues.

One approach to finding the remaining graphs $G$ on $6$ vertices with $q(G)=2$ is to find graphs $G$ with few edges that admit a matrix with $q(G)=2$ that have the Strong Spectral Property. Three such graphs are given in Figure \ref{6VtxSSP}.
\begin{figure}[h!]
\centering
\begin{tikzpicture}[scale=0.65, vrtx/.style args = {#1/#2}{%
      circle, draw, fill=black, inner sep=0pt,
      minimum size=6pt, label=#1:#2}]
\node (0) [vrtx=left/] at (0,2) {};
\node (1) [vrtx=left/] at (0,0) {};
\node (2) [vrtx=left/] at (-1.25,1) {};
\node (3) [vrtx=left/] at (-2.75,1) {};
\node (4) [vrtx=left/] at (-4,0) {};
\node (5) [vrtx=left/] at (-4,2) {};
\node (l1) [] at (-2,-0.75) {$S_1$};
\draw (0) edge (1);
\draw (2) edge (0);
\draw (0) edge (5);
\draw (1) edge (2);
\draw (1) edge (4);
\draw (2) edge (3);
\draw (3) edge (5);
\draw (4) edge (3);
\draw (4) edge (5);
\node (01) [vrtx=left/] at (2+3.5,1) {};
\node (11) [vrtx=left/] at (0.75+3.5,1) {};
\node (21) [vrtx=left/] at (-0.75+3.5,0) {};
\node (31) [vrtx=left/] at (-0.75+3.5,1) {};
\node (41) [vrtx=left/] at (-0.75+3.5,2) {};
\node (51) [vrtx=left/] at (-2+3.5,1) {};
\node (l2) [] at (0+3,-0.75) {$S_2$};
\draw (01) edge (21);
\draw (01) edge (41);
\draw (11) edge (21);
\draw (11) edge (31);
\draw (11) edge (41);
\draw (51) edge (21);
\draw (51) edge (31);
\draw (51) edge (41);
\draw (21) edge (31);
\draw (41) edge (31);
\node (02) [vrtx=left/] at (2+9,1) {};
\node (12) [vrtx=left/] at (0.75+9,1) {};
\node (22) [vrtx=left/] at (0+9,0) {};
\node (32) [vrtx=left/] at (0+9,2) {};
\node (42) [vrtx=left/] at (-0.75+9,1) {};
\node (52) [vrtx=left/] at (-2+9,1) {};
\node (l3) [] at (0+9,-0.75) {$S_3$};
\draw (02) edge (12);
\draw (02) edge (22);
\draw (02) edge (32);
\draw (12) edge (32);
\draw (12) edge (22);
\draw (42) edge (22);
\draw (42) edge (32);
\draw (52) edge (22);
\draw (52) edge (32);
\draw (52) edge (42);
\end{tikzpicture}
\caption{Graphs $S_1$, $S_2$, and $S_3$.}\label{6VtxSSP}
\end{figure}
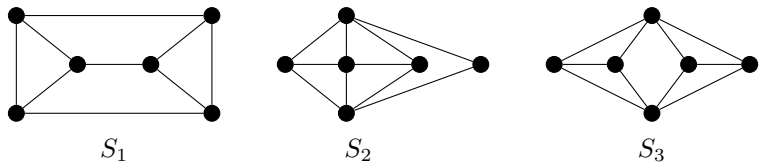
Each of $S_1$, $S_2$, and $S_3$ admits a matrix with $2$ eigenvalues and the Strong Spectral Property. Of the $23$ connected graphs $G$ on $6$ vertices with at least $9$ edges and $q(G)=2$, $20$ are supergraphs of one of $S_1$, $S_2$, or $S_3$. The remaining $3$ graphs are $G_3^{(1)}$, $G_3^{(3)}$, and the graph $G_3^{(4)}$ shown in Figure \ref{lastone}.
\begin{figure}[h!]
\centering
\begin{tikzpicture}[scale=0.65, vrtx/.style args = {#1/#2}{%
      circle, draw, fill=black, inner sep=0pt,
      minimum size=6pt, label=#1:#2}]
\node (0) [vrtx=left/] at (2,0) {};
\node (1) [vrtx=left/] at (1,1.73) {};
\node (2) [vrtx=left/] at (-1,1.73) {};
\node (3) [vrtx=left/] at (-2,0) {};
\node (4) [vrtx=left/] at (-1,-1.73) {};
\node (5) [vrtx=left/] at (1,-1.73) {};

\node (l1) [] at (0,-2.5) {$G_3^{(4)}$};
\draw (0) edge (1);
\draw (2) edge (1);
\draw (2) edge (3);
\draw (4) edge (3);
\draw (4) edge (5);
\draw (0) edge (5);
\draw (2) edge (5);
\draw (4) edge (1);
\draw (1) edge (5);
\draw (2) edge (4);
\end{tikzpicture}
\caption{The graph $G_3^{(4)}$.}\label{lastone}
\end{figure}
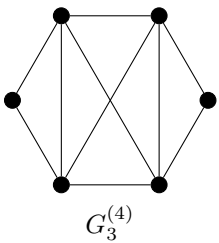
Each of $G_3^{(1)}$, $G_3^{(3)}$, and $G_3^{(4)}$ has a supergraph with $q$-value greater than $2$. So none of them admit a matrix with $2$ distinct eigenvalues and the Strong Spectral Property. Interestingly, all three are supergraphs of $G_3$. Each of $G_3^{(1)}$, $G_3^{(3)}$, and $G_3^{(4)}$ can be shown to allow two distinct eigenvalues by finding a compatible matrix with only $2$ distinct eigenvalues. Currently, we have no other way to establish this. 

Further investigation of the $q$-values of the supergraphs of $G_k$ and $G'_k$ for larger values of $k$ seems called for. For example, if $G$ is obtained from $G_k$, or $G_k'$, by adding any subset of the non-edges between the pairs of vertices in $N_i(1)$, what is $q(G)$? The graphs in Figure \ref{G3andG4withedge} both have $q$-value equal to $2$; do all such $G$ have $q(G)=2$?

\section*{Acknowledgements}

Shaun M. Fallat was supported in part by an NSERC Discovery Research Grant, Application No.: RGPIN--2019--03934. Michael Tait was supported in part by NSF grant DMS-2011553.

This project began as part of the ``Inverse Eigenvalue Problems for Graphs and Zero Forcing” Research Community sponsored by the American Institute of Mathematics (AIM). We thank AIM for their support, and we thank the organizers and participants for contributing to this stimulating research experience.

We thank Bryan Shader for very helpful discussions and insights on this project.

\bibliographystyle{plain}
\bibliography{bibliography}

\end{document}